\documentclass[12pt,a4paper]{amsart}
\allowdisplaybreaks[1]

\usepackage{amstext}
\usepackage{amsthm}
\usepackage{amssymb}
\usepackage{bm}
\usepackage{ytableau}
\usepackage{mathdots}

\usepackage{tikz}
\usepackage[all]{xy}
\usepackage{graphics}
\usepackage[top=30truemm,bottom=30truemm,left=25truemm,right=25truemm]{geometry}

\DeclareFontEncoding{OT2}{}{}
\DeclareFontSubstitution{OT2}{cmr}{n}{l}

\DeclareFontFamily{OT2}{cmr}{\hyphenchar\font45}
\DeclareFontShape{OT2}{cmr}{n}{l}{%
<5><6><7><8><9>gen*wncyr%
<10><10.95><12><14.4><17.28><20.74><24.88>wncyr10}{}

\DeclareMathAlphabet{\mathcyr}{OT2}{cmr}{n}{l}

\newtheorem{thm}{Theorem}[section]
\newtheorem*{thm*}{Theorem}
\newtheorem{lem}[thm]{Lemma}

\newtheorem{cor}[thm]{Corollary}

\theoremstyle{definition}
\newtheorem{defn}[thm]{Definition}
\newtheorem{ex}[thm]{Example}
\theoremstyle{remark}
\newtheorem{rem}[thm]{Remark}

\DeclareMathOperator{\up}{\overset{\scalebox{0.8}[0.4]{$\diagup$}}{}}
\DeclareMathOperator{\down}{\overset{\scalebox{0.8}[0.4]{$\diagdown$}}{}}

\DeclareMathOperator{\open}{\substack{\scalebox{0.8}[0.4]{$\diagup$}\\\scalebox{0.8}[0.4]{$\diagdown$}}}
\DeclareMathOperator{\close}{\substack{\scalebox{0.8}[0.4]{$\diagdown$}\\\scalebox{0.8}[0.4]{$\diagup$}}}

\usepackage[T1]{fontenc}
\usepackage[latin9]{inputenc}
\usepackage{mathtools}
\usepackage{amsthm}
\usepackage{amssymb}

\begin{document}

\title{Integral expressions for Schur multiple zeta values}

\author{Minoru Hirose}
\address[Minoru Hirose]{Institute for Advanced Research, Nagoya University,  Furo-cho, Chikusa-ku, Nagoya, 464-8602, Japan}
\email{minoru.hirose@math.nagoya-u.ac.jp}

\author{Hideki Murahara}
\address[Hideki Murahara]{The University of Kitakyushu,  4-2-1 Kitagata, Kokuraminami-ku, Kitakyushu, Fukuoka, 802-8577, Japan}
\email{hmurahara@mathformula.page}

\author{Tomokazu Onozuka}
\address[Tomokazu Onozuka]{Institute of Mathematics for Industry, Kyushu University 744, Motooka, Nishi-ku, Fukuoka, 819-0395, Japan} \email{t-onozuka@imi.kyushu-u.ac.jp}


\begin{abstract}
Nakasuji, Phuksuwan, and Yamasaki defined the Schur multiple zeta values and gave iterated integral expressions of the Schur multiple zeta values of the ribbon type. 
This paper generalizes their integral expressions to the ones of more general Schur multiple zeta values having constant entries on the diagonals. Furthermore, we also discuss the duality relations for Schur multiple zeta values obtained from the integral expressions.
\end{abstract}

\maketitle

\section{Introduction}
\subsection{Schur multiple zeta values}
Let $\mathbb{N}=\{1,2,3,\dots\}$ be the set of positive integers. A
Young diagram is a finite subset $\lambda$ of $\mathbb{N}^{2}$ satisfying
\[
(i+1,j)\in\lambda\text{ or }(i,j+1)\in\lambda\Rightarrow(i,j)\in\lambda
\]
for $(i,j)\in\mathbb{N}^{2}$. Furthermore, we say that $\lambda\subset\mathbb{N}^{2}$
is a skew Young diagram if there exist Young diagrams $\lambda'$
and $\lambda''$ such that $\lambda'\subset\lambda''$ and $\lambda''\setminus\lambda'=\lambda$.
Hereafter, we fix a skew Young diagram $\lambda$. We define the set of semi-standard
Young tableaux $\mathrm{SSYT}(\lambda)$ as the set of maps $f$ from
$\lambda$ to $\mathbb{N}$ satisfying the conditions:
\begin{itemize}
\item If $(i,j),(i,j+1)\in\lambda$ then $f(i,j)\leq f(i,j+1)$,
\item If $(i,j),(i+1,j)\in\lambda$ then $f(i,j)<f(i+1,j)$.
\end{itemize}
We call a map $\boldsymbol{k}:\lambda\to\mathbb{N}$ an index for
$\lambda$. Furthermore, we say that $\bm{k}$ is admissible if $\bm{k}(i,j)\geq2$
for any corner $(i,j)\in\lambda$, and denote by $T(\lambda)$ the
set of admissible indices for $\lambda$. Here, a corner means an
element $(i,j)$ of $\lambda$ such that $(i+1,j)\notin\lambda$ and
$(i,j+1)\notin\lambda$. For an admissible index $\bm{k}$, the Schur
multiple zeta value (SMZV) \cite{NPY18} is defined by
\[
\zeta(\bm{k})=\sum_{f\in\mathrm{SSYT}(\lambda)}\frac{1}{\prod_{(i,j)\in\lambda}f(i,j)^{\bm{k}(i,j)}}.
\]
A skew Young diagram is depicted as the collection of boxes and 
an index for a skew Young diagram is depicted as the collection of boxes filled
with numbers (see Example \ref{example}).
Note that SMZVs are the common generalization of multiple zeta and zeta-star values
\begin{align*}
 \zeta(k_1,\dots, k_r)
 =\sum_{1\le n_1<\cdots <n_r} \frac{1}{n_1^{k_1}\cdots n_r^{k_r}}, \quad
 \zeta^\star (k_1,\dots, k_r)
 =\sum_{1\le n_1\le \cdots \le n_r} \frac{1}{n_1^{k_1}\cdots n_r^{k_r}}.
\end{align*} 
\begin{ex}\label{example}
 Let
 $\lambda=\{(1,2),(1,3),(2,1),(2,2)\}$, then
 $\lambda$  is depicted as
 $
 \ {\footnotesize \ytableausetup{centertableaux, boxsize=1.2em}
 	\begin{ytableau}
 	 \none & {} & {}\\
 	 {} & {} 
 	\end{ytableau}}\
 $. 
 By putting $
 \boldsymbol{k}=
 \ {\footnotesize \ytableausetup{centertableaux, boxsize=1.2em}
 	\begin{ytableau}
 	 \none & {1} & {3}\\
 	 {4} & {2} 
 	\end{ytableau}}\
 $, we find
 \begin{align*}
  \zeta\left(
  \ {\footnotesize \ytableausetup{centertableaux, boxsize=1.2em}
   \begin{ytableau}
    \none & 1 & 3 \\
 	4 & 2  & \none
   \end{ytableau}}\
  \right)
  =\sum_{\substack{ {\;\;\;\;\;\;\;\; } \; \quad n_{1,2}\le n_{1,3} \\  \quad \text{\;\rotatebox{90}{$\quad$}} 
   \quad\; \text{\rotatebox{90}{$>$}}\qquad \\ \quad n_{2,1} \le n_{2,2}\qquad }}
   \frac{1}{ n_{1,2} n_{1,3}^3n_{2,1}^4 n_{2,2}^{2} }.
 \end{align*}
\end{ex}

We say that $\lambda$ is connected if all boxes are edge-connected,
i.e., there is only one $(i,j)\in\lambda$
satisfying $(i-1,j)\notin\lambda$ and $(i,j+1)\notin\lambda$.
In this paper, we only consider the case where $\lambda$ is connected for simplicity.

\subsection{Yamamoto's integral expression}
Yamamoto used 2-labeled posets to represent iterated integrals of multiple zeta and zeta-star values (see  \cite{Yam17}).
We call a pair $(X,\delta_X)$ 2-labeled poset if $X=(X,\le)$ is a finite partially ordered set and $\delta_X$ is a map from $X$ to $\{\circ, \bullet\}$.
A 2-labeled poset $(X, \delta_X)$ is called admissible if $\delta_X(p)=\circ$ for all maximal elements $p\in X$ and $\delta_X(q)=\bullet$ for all minimal elements $q\in X$. 
Sometimes, a 2-labeled poset is depicted by a Hasse diagram.
For example, the diagram
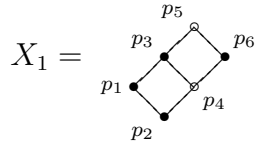
\begin{align*}
X_1=
\begin{xy}
{(8,-4) \ar @{{*}}^{p_1}},
{(8,-4) \ar @{-} (12,0)},
{(12,0) \ar @{{*}}^{p_3}},
{(8,-4) \ar @{-} (12,-8)},
{(12,-8) \ar @{{*}}_{p_2}},
{(12,0) \ar @{-} (16,4)},
{(16,4) \ar @{o}^{p_5}},
{(12,0) \ar @{-} (16,-4)},  
{(12,-8) \ar @{-} (16,-4)},  
{(16,-4) \ar @{o}_{p_4}},
{(16,4) \ar @{-} (20,0)}, 
{(20,0) \ar @{{*}}^{p_6}},
{(16,-4) \ar @{-} (20,0)}
\end{xy}\ 
\end{align*}
represents the admissible 2-labeled poset $X_1 = \{p_1, \ldots, p_6\}$ with order $p_6>p_4>p_2<p_1<p_3<p_5>p_6$ and $p_3>p_4$
and $(\delta_{X_1} (p_1),\ldots,\delta_{X_1} (p_6)) = (\bullet, \bullet, \bullet, \circ, \circ, \bullet)$.

We denote by $\mathcal{U}$ the set of closed intervals contained in
$[0,1]$.
For any poset $X$ and $U=(U_{p})_{p\in X}\in\mathcal{U}^{X}$, we set
\begin{align*}
 \Delta_{X,U}=\{(t_p)_{p\in X} \in \prod_{p\in X}U_{p} \mid t_p<t_q \text{ if } p<q\}.
\end{align*}
Then, for an admissible 2-labeled poset $(X,\delta_X)$ and $U=(U_{p})_{p\in X}\in\mathcal{U}^{X}$, the associated integral  is  defined by
\begin{align*}
 I(X,U):=\int_{\Delta_{X,U}}\prod_{p \in X}\omega_{\delta_X(p)}(t_p),
\end{align*}
where
\begin{align*}
 \omega_\circ(t):=\frac{dt}{t} \;\textrm{ and }\; \omega_\bullet(t):=\frac{dt}{1-t}.
\end{align*}
Furthermore, we put
\[
I(X)\coloneqq I(X,([0,1])_{p\in X}).
\]
Note that the integral $I(X)$ converges if and only if the 2-labeled poset $X$ is admissible.
For example, we have
\begin{align*}
 I(X_1)
 =\int_{ \substack{t_1,\ldots,t_6\in(0,1)\\ t_{6}>t_{4}>t_{2}<t_{1}<t_{3}<t_{5}>t_{6}\\ t_{3}>t_{4}}}
  \frac{dt_{1}}{1-t_{1}}\frac{dt_{2}}{1-t_{2}}\frac{dt_{3}}{1-t_{3}}
  \frac{dt_{4}}{t_{4}}\frac{dt_{5}}{t_{5}}\frac{dt_{6}}{1-t_{6}}.
\end{align*}

\subsection{Main theorem}
Put 
\[
T^{\mathrm{diag}}(\lambda)=\{\bm{k}\in T(\lambda):\bm{k}(i,j)=\bm{k}(i+1,j+1)\text{ if }(i,j)\in\lambda\text{ and }(i+1,j+1)\in\lambda\}.
\]
We say that $\lambda$ is ribbon type if $T^{\mathrm{diag}}(\lambda)=T(\lambda)$.
Nakasuji, Phuksuwan, and Yamasaki gave Yamamoto's integral expression for the ribbon-type SMZVs in \cite{NPY18}.
The following is the main theorem of this paper, which generalizes the integral expression for the ribbon-type SMZVs.
See Theorem \ref{main2} for a more rigorous formulation of the theorem.

\begin{thm} \label{main1}
 If $\boldsymbol{k}\in T^{\mathrm{diag}}(\lambda)$, then the SMZV 
 $\zeta(\boldsymbol{k})$ has Yamamoto's integral expression, i.e., we have
  \begin{align*}
  \zeta\left(
  \ {\footnotesize \ytableausetup{centertableaux, boxsize=1.2em}
  \begin{ytableau}
  \none & \none & \none & \none &  k_1   \\
  \none &  k_i & \none & \none[\iddots] & \none \\
  \none & \none & \none[\ddots] \\
  \none & \none[\iddots] & \none & k_i  \\
  k_r & \none & \none
 \end{ytableau}}\
\right)=I\left(\ 
\begin{xy}
{(0,-6) \ar @{{*}-o} (4,-2)}, 
{(4,-2) \ar @{.o} (8,2)}, 
{(0,-5) \ar @/^2mm/ @{-}^{k_1} (7,2)}, 
{(8,2) \ar @{.} (12,6)},  
{(8,2) \ar @{.} (12,-2)}
\end{xy}
\cdots
\begin{xy}
{(32,4) \ar @{{*}-o} (36,8)},   
{(32,4) \ar @{{*}-o} (36,0)},   
{(32,-4) \ar @{{*}-} (36,0)},   
{(32,-4) \ar @{.} (36,-8)},   
{(32,-12) \ar @{{*}-o} (36,-8)},   
{(32,4) \ar @/_2mm/ @{-}_{l_i} (32,-11)}, 
{(36,8) \ar @{.o} (40,12)},   
{(36,8) \ar @{.o} (40,4)},   
{(36,0) \ar @{.} (40,4)},   
{(36,0) \ar @{.} (40,-4)},   
{(36,-8) \ar @{.o} (40,-4)}, 
{(32,5) \ar @/^2mm/ @{-}^{k_i} (39,12)}, 
\end{xy}\ 
\cdots
\begin{xy}
{(44,2) \ar @{.} (48,-2)},   
{(44,-6) \ar @{.{*}} (48,-2)},   
{(48,-2) \ar @{-o} (52,2)}, 
{(52,2) \ar @{.o} (56,6)}, 
{(48,-1) \ar @/^2mm/ @{-}^{k_r} (55,6)}, 
\end{xy}\ 
\right),
\end{align*}
where $l_i$ is the number of $k_i$'s.
\end{thm}
\begin{rem}
This remark gives precise explanations of Theorem \ref{main1}.
There are four cases to connect adjacent cells  of the skew Young diagram.
Looking at these four cases, we can clearly describe the correspondence between the SMZV and its iterated integral expression: 
\begin{align*}
\zeta\left(
\ {\footnotesize \ytableausetup{centertableaux, boxsize=1.2em}
	\begin{ytableau}
	\none & \none & \none & \none &  k_1   \\
	\none &  a & \none & \none[\iddots] & \none \\
	\none & b & \none[\ddots] \\
	\none & \none & \none[\ddots] & a  \\
	\none     & \none[\iddots] & \none & b \\
	k_r
	\end{ytableau}}\
\right)=I\left(\ 
\begin{xy}
{(0,-6) \ar @{{*}-o} (4,-2)}, 
{(4,-2) \ar @{.o} (8,2)}, 
{(0,-5) \ar @/^2mm/ @{-}^{k_1} (7,2)}, 
{(8,2) \ar @{.} (12,6)},  
{(8,2) \ar @{.} (12,-2)}
\end{xy}
\cdots\ 
\begin{xy}   
{(32,-2) \ar @{{*}-o} (36,2)},   
{(32,-2) \ar @{{*}-o} (36,-6)},   
{(32,-10) \ar @{{*}-} (36,-6)},   
{(32,-10) \ar @{.} (36,-14)},   
{(32,-18) \ar @{{*}-o} (36,-14)},   
%
{(36,2) \ar @{.o} (40,6)},   
{(36,2) \ar @{.o} (40,-2)},   
{(36,-6) \ar @{.} (40,-2)},   
{(36,-6) \ar @{.} (40,-10)},   
{(36,-14) \ar @{.o} (40,-10)}, 
{(32,-1) \ar @/^2mm/ @{-}^{a} (39,6)}, 
{(40,6) \ar @{-{*}} (44,10)},  
{(40,6) \ar @{-{*}} (44,2)},   
{(40,-2) \ar @{-} (44,2)},   
{(40,-2) \ar @{.} (44,-6)},   
{(40,-10) \ar @{-{*}} (44,-6)},
{(44,10) \ar @{{*}-o} (48,14)},   
{(44,10) \ar @{{*}-o} (48,6)},   
{(44,2) \ar @{{*}-} (48,6)},   
{(44,2) \ar @{.} (48,-2)},   
{(44,-6) \ar @{{*}-o} (48,-2)},   
{(48,14) \ar @{.o} (52,18)},   
{(48,14) \ar @{.o} (52,10)},   
{(48,6) \ar @{.} (52,10)},   
{(48,6) \ar @{.} (52,2)},   
{(48,-2) \ar @{.o} (52,2)}, 
{(44,11) \ar @/^2mm/ @{-}^{b} (51,18)}, 
\end{xy}\ 
\cdots
\begin{xy}
{(44,2) \ar @{.} (48,-2)},   
{(44,-6) \ar @{.{*}} (48,-2)},   
{(48,-2) \ar @{-o} (52,2)}, 
{(52,2) \ar @{.o} (56,6)}, 
{(48,-1) \ar @/^2mm/ @{-}^{k_r} (55,6)}, 
\end{xy}\ 
\right),
\end{align*}

 \begin{align*}
\zeta\left(
\ {\footnotesize \ytableausetup{centertableaux, boxsize=1.2em}
	\begin{ytableau}
	\none & \none & \none & \none & \none &  k_1   \\
	\none & b &  a & \none & \none[\iddots] & \none \\
	\none & \none & \none[\ddots] & \none[\ddots] \\
	\none & \none & \none & \none[\ddots] & a  \\
	\none     & \none[\iddots] & \none & \none & b \\
	k_r
	\end{ytableau}}\
\right)=I\left(\ 
\begin{xy}
{(0,-6) \ar @{{*}-o} (4,-2)}, 
{(4,-2) \ar @{.o} (8,2)}, 
{(0,-5) \ar @/^2mm/ @{-}^{k_1} (7,2)}, 
{(8,2) \ar @{.} (12,6)},  
{(8,2) \ar @{.} (12,-2)}
\end{xy}
\cdots\ 
\begin{xy}   
{(32,-2) \ar @{{*}-o} (36,2)},   
{(32,-2) \ar @{{*}-o} (36,-6)},   
{(32,-10) \ar @{{*}-} (36,-6)},   
{(32,-10) \ar @{.} (36,-14)},   
{(32,-18) \ar @{{*}-o} (36,-14)},   
%
{(36,2) \ar @{.o} (40,6)},   
{(36,2) \ar @{.o} (40,-2)},   
{(36,-6) \ar @{.} (40,-2)},   
{(36,-6) \ar @{.} (40,-10)},   
{(36,-14) \ar @{.o} (40,-10)}, 
{(32,-1) \ar @/^2mm/ @{-}^{a} (39,6)}, 
{(40,6) \ar @{-{*}} (44,10)},  
{(40,6) \ar @{-{*}} (44,2)},   
{(40,-2) \ar @{-} (44,2)},   
{(40,-2) \ar @{.} (44,-6)},   
{(40,-10) \ar @{-{*}} (44,-6)},
{(40,-10) \ar @{-{*}} (44,-14)},
{(44,10) \ar @{{*}-o} (48,14)},   
{(44,10) \ar @{{*}-o} (48,6)},   
{(44,2) \ar @{{*}-} (48,6)},   
{(44,2) \ar @{.} (48,-2)},   
{(44,-6) \ar @{{*}-o} (48,-2)},   
{(44,-6) \ar @{-} (48,-10)},   
{(44,-14) \ar @{{*}-o} (48,-10)},   
{(48,14) \ar @{.o} (52,18)},   
{(48,14) \ar @{.o} (52,10)},   
{(48,6) \ar @{.} (52,10)},   
{(48,6) \ar @{.} (52,2)},   
{(48,-2) \ar @{.o} (52,2)}, 
{(48,-2) \ar @{.} (52,-6)},   
{(48,-10) \ar @{.o} (52,-6)}, 
{(44,11) \ar @/^2mm/ @{-}^{b} (51,18)}, 
\end{xy}\ 
\cdots
\begin{xy}
{(44,2) \ar @{.} (48,-2)},   
{(44,-6) \ar @{.{*}} (48,-2)},   
{(48,-2) \ar @{-o} (52,2)}, 
{(52,2) \ar @{.o} (56,6)}, 
{(48,-1) \ar @/^2mm/ @{-}^{k_r} (55,6)}, 
\end{xy}\ 
\right),
\end{align*}

 \begin{align*}
\zeta\left(
\ {\footnotesize \ytableausetup{centertableaux, boxsize=1.2em}
	\begin{ytableau}
	\none & \none & \none & \none & \none &  k_1   \\
	\none & b &  a & \none & \none[\iddots] & \none \\
	\none & \none & \none[\ddots] & \none[\ddots] \\
	\none & \none[\iddots] & \none & b & a  \\
	k_r     & \none & \none & \none  
	\end{ytableau}}\
\right)=I\left(\ 
\begin{xy}
{(0,-6) \ar @{{*}-o} (4,-2)}, 
{(4,-2) \ar @{.o} (8,2)}, 
{(0,-5) \ar @/^2mm/ @{-}^{k_1} (7,2)}, 
{(8,2) \ar @{.} (12,6)},  
{(8,2) \ar @{.} (12,-2)}
\end{xy}
\cdots\ 
\begin{xy}   
{(32,2) \ar @{{*}-o} (36,6)},   
{(32,2) \ar @{{*}-o} (36,-2)},   
{(32,-6) \ar @{{*}-} (36,-2)},   
{(32,-6) \ar @{.} (36,-10)},   
{(32,-14) \ar @{{*}-o} (36,-10)},   
%
{(36,6) \ar @{.o} (40,10)},   
{(36,6) \ar @{.o} (40,2)},   
{(36,-2) \ar @{.} (40,2)},   
{(36,-2) \ar @{.} (40,-6)},   
{(36,-10) \ar @{.o} (40,-6)}, 
{(32,3) \ar @/^2mm/ @{-}^{a} (39,10)}, 
{(40,10) \ar @{-{*}} (44,6)},   
{(40,2) \ar @{-} (44,6)},   
{(40,2) \ar @{.} (44,-2)},   
{(40,-6) \ar @{-{*}} (44,-2)},
{(40,-6) \ar @{-{*}} (44,-10)},
{(44,6) \ar @{{*}-o} (48,10)},   
{(44,6) \ar @{.} (48,2)},   
{(44,-2) \ar @{{*}-o} (48,2)},   
{(44,-2) \ar @{-} (48,-6)},   
{(44,-10) \ar @{{*}-o} (48,-6)},   
{(48,10) \ar @{.o} (52,14)},   
{(48,10) \ar @{.} (52,6)},   
{(48,2) \ar @{.o} (52,6)}, 
{(48,2) \ar @{.} (52,-2)},   
{(48,-6) \ar @{.o} (52,-2)}, 
{(44,7) \ar @/^2mm/ @{-}^{b} (51,14)}, 
\end{xy}\ 
\cdots
\begin{xy}
{(44,2) \ar @{.} (48,-2)},   
{(44,-6) \ar @{.{*}} (48,-2)},   
{(48,-2) \ar @{-o} (52,2)}, 
{(52,2) \ar @{.o} (56,6)}, 
{(48,-1) \ar @/^2mm/ @{-}^{k_r} (55,6)}, 
\end{xy}\ 
\right),
\end{align*}

\begin{align*}
\zeta\left(
\ {\footnotesize \ytableausetup{centertableaux, boxsize=1.2em}
	\begin{ytableau}
	\none & \none & \none & \none & \none & k_1   \\
	\none &  a & \none & \none & \none[\iddots] & \none \\
	\none & b & \none[\ddots] \\
	\none & \none & \none[\ddots] &  \none[\ddots] \\
	\none     & \none[\iddots] & \none & b & a \\
	k_r
	\end{ytableau}}\
\right)=I\left(\ 
\begin{xy}
{(0,-6) \ar @{{*}-o} (4,-2)}, 
{(4,-2) \ar @{.o} (8,2)}, 
{(0,-5) \ar @/^2mm/ @{-}^{k_1} (7,2)}, 
{(8,2) \ar @{.} (12,6)},  
{(8,2) \ar @{.} (12,-2)}
\end{xy}
\cdots\ 
\begin{xy}   
{(32,2) \ar @{{*}-o} (36,6)},   
{(32,2) \ar @{{*}-o} (36,-2)},   
{(32,-6) \ar @{{*}-} (36,-2)},   
{(32,-6) \ar @{.} (36,-10)},   
{(32,-14) \ar @{{*}-o} (36,-10)},   
%
{(36,6) \ar @{.o} (40,10)},   
{(36,6) \ar @{.o} (40,2)},   
{(36,-2) \ar @{.} (40,2)},   
{(36,-2) \ar @{.} (40,-6)},   
{(36,-10) \ar @{.o} (40,-6)}, 
{(32,3) \ar @/^2mm/ @{-}^{a} (39,10)}, 
{(40,10) \ar @{-{*}} (44,6)},   
{(40,2) \ar @{-} (44,6)},   
{(40,2) \ar @{.} (44,-2)},   
{(40,-6) \ar @{-{*}} (44,-2)},
{(44,6) \ar @{{*}-o} (48,10)},   
{(44,6) \ar @{.} (48,2)},   
{(44,-2) \ar @{{*}-o} (48,2)},   
{(48,10) \ar @{.o} (52,14)},   
{(48,10) \ar @{.} (52,6)},   
{(48,2) \ar @{.o} (52,6)}, 
{(44,7) \ar @/^2mm/ @{-}^{b} (51,14)}, 
\end{xy}\ 
\cdots
\begin{xy}
{(44,2) \ar @{.} (48,-2)},   
{(44,-6) \ar @{.{*}} (48,-2)},   
{(48,-2) \ar @{-o} (52,2)}, 
{(52,2) \ar @{.o} (56,6)}, 
{(48,-1) \ar @/^2mm/ @{-}^{k_r} (55,6)}, 
\end{xy}\ 
\right).
\end{align*}
\end{rem} 

For better understanding, let us make one note. 
The expression on the right-hand side of Theorem \ref{main1} might be written as
\begin{align*}
I\left(\ 
\begin{xy}
{(0,-6) \ar @{{*}-o} (4,-2)}, 
{(4,-2) \ar @{.o} (8,2)}, 
{(0,-5) \ar @/^2mm/ @{-}^{k_1} (7,2)}, 
{(8,2) \ar @{.} (12,6)},  
{(8,2) \ar @{.} (12,-2)}
\end{xy}
\cdots
\begin{xy}
{(32,4) \ar @{{*}-o} (36,8)},   
{(32,4) \ar @{{*}-o} (36,0)},   
{(32,-4) \ar @{{*}-} (36,0)},   
{(32,-4) \ar @{.} (36,-8)},   
{(32,-12) \ar @{{*}-o} (36,-8)},   
{(32,4) \ar @/_2mm/ @{-}_{l_i} (32,-11)}, 
{(36,8) \ar @{.o} (40,12)},   
{(36,8) \ar @{.o} (40,4)},   
{(36,0) \ar @{.} (40,4)},   
{(36,0) \ar @{.} (40,-4)},   
{(36,-8) \ar @{.o} (40,-4)}, 
{(32,5) \ar @/^2mm/ @{-}^{k_i} (39,12)}, 
\end{xy}\ 
\cdots
\begin{xy}
{(44,2) \ar @{.} (48,-2)},   
{(44,-6) \ar @{.{*}} (48,-2)},   
{(48,-2) \ar @{-o} (52,2)}, 
{(52,2) \ar @{.o} (56,6)}, 
{(48,-1) \ar @/^2mm/ @{-}^{k_r} (55,6)}, 
\end{xy}\ 
\right)
=
I \left(
\rotatebox{45}{
\raisebox{-2ex}{\footnotesize \ytableausetup{centertableaux, boxsize=1.3em}
	\begin{ytableau}
	 \raisebox{0.6ex}{\rotatebox{-45}{$k_1$}} 
	\end{ytableau}}\
}
\cdots\hspace{-6ex}
\rotatebox{-45}{
\raisebox{5.5ex}{\footnotesize \ytableausetup{centertableaux, boxsize=1.3em}
  \begin{ytableau}
    \raisebox{-0.2ex}{\rotatebox{45}{$k_i$}}  & \none & \none & \none \\
   \none & \none[\raisebox{1.2ex}{\rotatebox{-45}{$\cdots$}}]  \\
   \none & \none & \raisebox{-0.2ex}{\rotatebox{45}{$k_i$}}  
 \end{ytableau}}\
}
\hspace{-6ex}\cdots\hspace{-0.5ex}
\rotatebox{45}{
\raisebox{-2ex}{\footnotesize \ytableausetup{centertableaux, boxsize=1.3em}
	\begin{ytableau}
	 \raisebox{0.6ex}{\rotatebox{-45}{$k_r$}} 
	\end{ytableau}}\
}
\right),
\end{align*}
which is obtained by rotating the corresponding tableau index by $3\pi/4$ counterclockwise.

Here we give some examples. 
\begin{ex}\label{example1}
We have
\begin{align*}
\zeta\left(
\ {\footnotesize \ytableausetup{centertableaux, boxsize=1.2em}
	\begin{ytableau}
	 2 & 1 \\
	 1 & 2 
	\end{ytableau}}\
\right)
&=I
\left(
\hspace{-0.9ex}
\rotatebox{45}{
\raisebox{-4ex}{\footnotesize \ytableausetup{centertableaux, boxsize=1.3em}
	\begin{ytableau}
	 \raisebox{0.3ex}{\rotatebox{-45}{1}}  & \raisebox{0.3ex}{\rotatebox{-45}{2}}  \\
	 \raisebox{0.3ex}{\rotatebox{-45}{2}}  & \raisebox{0.3ex}{\rotatebox{-45}{1}}  
	\end{ytableau}}\
}
\right)
=I\left(\ 
\begin{xy}
{(8,-2) \ar @{{*}}},
{(8,-2) \ar @{-} (12,2)},
{(12,2) \ar @{{*}}},
{(8,-2) \ar @{-} (12,-6)},
{(12,-6) \ar @{{*}}},
{(12,2) \ar @{-} (16,6)},
{(16,6) \ar @{o}},
{(12,2) \ar @{-} (16,-2)},  
{(12,-6) \ar @{-} (16,-2)},  
{(16,-2) \ar @{o}},
{(16,6) \ar @{-} (20,2)}, 
{(20,2) \ar @{{*}}},
{(16,-2) \ar @{-} (20,2)}
\end{xy} \ 
\right),\\
\zeta\left(
\ {\footnotesize \ytableausetup{centertableaux, boxsize=1.2em}
	\begin{ytableau}
	 3 &  4 \\
	 2 & 3 
	\end{ytableau}}\
\right)
&=I
\left(
\hspace{-0.9ex}
\rotatebox{45}{
\raisebox{-3.5ex}{\footnotesize \ytableausetup{centertableaux, boxsize=1.3em}
	\begin{ytableau}
	 \raisebox{0.3ex}{\rotatebox{-45}{4}}  & \raisebox{0.3ex}{\rotatebox{-45}{3}}  \\
	 \raisebox{0.3ex}{\rotatebox{-45}{3}}  & \raisebox{0.3ex}{\rotatebox{-45}{2}}  
	\end{ytableau}}\
}
\right)
=I\left(\ 
\begin{xy}
{(0,-12) \ar @{{*}-o} (4,-8)}, 
{(4,-8) \ar @{-o} (8,-4)}, 
{(8,-4) \ar @{-o} (12,0)},
{(12,0) \ar @{-{*}} (16,4)},  
{(12,0) \ar @{-{*}} (16,-4)},  
{(16,4) \ar @{-o} (20,8)}, 
{(16,4) \ar @{-o} (20,0)},
{(16,-4) \ar @{-} (20,0)},  
{(20,8) \ar @{-o} (24,12)},
{(20,8) \ar @{-o} (24,4)}, 
{(20,0) \ar @{-} (24,4)}, 
{(24,12) \ar @{-{*}} (28,8)},
{(24,4) \ar @{-} (28,8)},   
{(28,8) \ar @{-o} (32,12)}
\end{xy}\ 
\right),\\
\end{align*}
\begin{align*}
\zeta\left(
\ {\footnotesize \ytableausetup{centertableaux, boxsize=1.2em}
	\begin{ytableau}
	\none & 2 & 4 \\
	1 & 3 & 2 \\
	2 & 1 & 3 
	\end{ytableau}}\
\right)
&=I
\left(
\hspace{-0.9ex}
\rotatebox{45}{
\raisebox{-5ex}{\footnotesize \ytableausetup{centertableaux, boxsize=1.3em}
	\begin{ytableau}
	 \raisebox{0.3ex}{\rotatebox{-45}{4}}  & \raisebox{0.3ex}{\rotatebox{-45}{2}} & \raisebox{0.3ex}{\rotatebox{-45}{3}}  \\
	 \raisebox{0.3ex}{\rotatebox{-45}{2}}  & \raisebox{0.3ex}{\rotatebox{-45}{3}} & \raisebox{0.3ex}{\rotatebox{-45}{1}}  \\
	 \none  & \raisebox{0.3ex}{\rotatebox{-45}{1}} & \raisebox{0.3ex}{\rotatebox{-45}{2}}  \\
	\end{ytableau}}\
}
\right)
=
I\left(\ 
\begin{xy}
{(0,-16) \ar @{{*}-o} (4,-12)}, 
{(4,-12) \ar @{-o} (8,-8)}, 
{(8,-8) \ar @{-o} (12,-4)},
{(12,-4) \ar @{-{*}} (16,0)},  
{(12,-4) \ar @{-{*}} (16,-8)}, 
{(16,0) \ar @{-o} (20,4)}, 
{(16,0) \ar @{-o} (20,-4)},
{(16,-8) \ar @{-} (20,-4)},  
{(20,4) \ar @{-{*}} (24,8)},
{(20,4) \ar @{-{*}} (24,0)}, 
{(20,-4) \ar @{-} (24,0)},
{(24,8) \ar @{-o} (28,12)},
{(24,8) \ar @{-o} (28,4)},
{(24,0) \ar @{-} (28,4)},   
{(28,12) \ar @{-o} (32,16)},
{(28,12) \ar @{-o} (32,8)},
{(28,4) \ar @{-} (32,8)},   
{(32,16) \ar @{-{*}} (36,12)},
{(32,8) \ar @{-} (36,12)},   
{(32,8) \ar @{-{*}} (36,4)}, 
{(36,12) \ar @{-{*}} (40,8)}, 
{(36,4) \ar @{-} (40,8)}, 
{(40,8) \ar @{-o} (44,12)}, 
\end{xy}\ 
\right),
\end{align*}
and
\begin{align*}
\zeta\left(
\ {\footnotesize \ytableausetup{centertableaux, boxsize=1.2em}
	\begin{ytableau}
	\none & \none & 2 \\
	\none & 2 & 4 \\
	2 & 3 & 2 \\
	1 & 2 & \none 
	\end{ytableau}}\
\right)
&=I
\left(
\hspace{-1.5ex}
\rotatebox{45}{
\raisebox{-6ex}{\footnotesize \ytableausetup{centertableaux, boxsize=1.3em}
	\begin{ytableau}
	 \raisebox{0.3ex}{\rotatebox{-45}{2}}  & \raisebox{0.3ex}{\rotatebox{-45}{4}} & \raisebox{0.3ex}{\rotatebox{-45}{2}} & \none  \\
	 \none & \raisebox{0.3ex}{\rotatebox{-45}{2}}  & \raisebox{0.3ex}{\rotatebox{-45}{3}} & \raisebox{0.3ex}{\rotatebox{-45}{2}}  \\
	 \none & \none & \raisebox{0.3ex}{\rotatebox{-45}{2}} & \raisebox{0.3ex}{\rotatebox{-45}{1}}  \\
	\end{ytableau}}\
}
\right)
=
I\left(\ 
\begin{xy}
{(-8,-20) \ar @{{*}-o} (-4,-16)}, 
{(-4,-16) \ar @{o-{*}} (0,-12)}, 
{(0,-12) \ar @{{*}-o} (4,-8)}, 
{(4,-8) \ar @{-o} (8,-4)}, 
{(8,-4) \ar @{-o} (12,0)},
{(12,0) \ar @{-{*}} (16,4)},  
{(12,0) \ar @{-{*}} (16,-4)}, 
{(16,4) \ar @{-o} (20,8)}, 
{(16,4) \ar @{-o} (20,0)},
{(16,-4) \ar @{-} (20,0)},  
{(20,8) \ar @{-{*}} (24,4)}, 
{(20,0) \ar @{-} (24,4)},
{(24,4) \ar @{-o} (28,8)}, 
{(28,8) \ar @{-o} (32,12)},   
{(32,12) \ar @{-{*}} (36,16)},   
{(32,12) \ar @{-{*}} (36,8)}, 
{(36,16) \ar @{-o} (40,20)},   
{(36,16) \ar @{-o} (40,12)}, 
{(36,8) \ar @{-} (40,12)}, 
{(40,20) \ar @{-{*}} (44,16)}, 
{(40,12) \ar @{-} (44,16)}, 
\end{xy}\ 
\right).
\end{align*}
\end{ex}

\section{Formal description of Theorem \ref{main1}}

\begin{defn}

For a symbol $r\in\{\up,\down,\close,\open\}$, we define $E(r)$ by
\[
E(\up)=E(\down)=0,\ E(\open)=1,\ E(\close)=-1.
\]
Furthermore, for symbols $r_{1},\dots,r_{s}\in\{\up,\down,\close,\open\}$,
we put $E(r_{1}\cdots r_{s})=E(r_{1})+\cdots+E(r_{s})$.
We say that $r_{1}\cdots r_{s}$ is weakly proper if $E(r_{1}\cdots r_{j})\geq0$
for all $1\leq j\leq s$, and $r_{1}\cdots r_{s}$ is proper if it
is weakly proper and $E(r_{1}\cdots r_{s})=0$.

\end{defn}

\begin{defn}

For symbols $r_{1},\dots,r_{s}\in\{\up,\down,\open,\close\}$ such that  $r_{1}\cdots r_{s}$ 
is weakly-proper, we 
define 
\[
L(r_1,\ldots,r_s):=\{(x,y)\in\mathbb{Z}^{2}:1\leq x\leq s+1,y\in Y_{x}\},
\]
where  $Y_{1},\dots,Y_{s+1}$ are subsets of
$\mathbb{Z}$ defined by $Y_{1}=\{0\}$ and
\[
Y_{x+1}=\begin{cases}
Y_{x}^{+} & r_{x}=\up,\\
Y_{x}^{-} & r_{x}=\down,\\
Y_{x}^{+}\cup Y_{x}^{-} & r_{x}=\open,\\
Y_{x}^{+}\cap Y_{x}^{-} & r_{x}=\close
\end{cases}
\]
with $Y_x^{\pm}=\{y\pm1:y\in Y_x\}$.

\end{defn}

\begin{defn}

Let $d_{1},\dots,d_{s+1}\in\{\circ,\bullet\}$ and $r_{1},\dots,r_{s}\in\{\up,\down,\open,\close\}$ such that $r_{1}\cdots r_{s}$ is weakly-proper.
We define a $2$-labeled poset $P(d_{1}r_{1}\cdots d_{s}r_{s}d_{s+1})$ as $((X,\leq),\delta_{X})$ where $X=L(r_1,\ldots,r_s)$,
 $\leq$ is the order\footnote{This order is equal to the one defined by
\[
(x,y)\leq(x',y')\Leftrightarrow y'-y\geq\left|x-x'\right|.
\]} generated by
\[
(x,y)\leq(x+1,y+1)\text{ and }(x,y)\geq(x+1,y-1),
\]
and 
\[
\delta_{X}:X\to\{\circ,\bullet\}
\]
is the map defined by $\delta_{X}((x,y))=d_{x}$.

\end{defn}

\begin{defn}\label{def_ref}

Let  $k_{1},\dots,k_{s+1}$ be positive integers and $r_{1},\dots,r_{s}\in\{\up,\down,\open,\close\}$ symbols such that $r_{1}\cdots r_{s}$ is proper.

\begin{enumerate}

\item We say that $k_{1}r_{1}\cdots k_{s}r_{s}k_{s+1}$ is admissible if and
only if $k_{i}>1$ for all $i$ satisfying
\[
(i=1\lor r_{i-1}\in\{\up,\open\})\land(i=s+1\lor r_{i}\in\{\down,\close\}).
\]
\item When $k_{1}r_{1}\cdots k_{s}r_{s}k_{s+1}$ is admissible, we define $\zeta(k_{1}r_{1}\cdots k_{s}r_{s}k_{s+1})\in\mathbb{R}$ by
\[
\zeta(k_{1}r_{1}\cdots k_{s}r_{s}k_{s+1})=\sum_{\substack{(m_{p})_{p\in L}\in\mathbb{N}^{L}\\
m_{(x,y)}<m_{(x+1,y+1)}\\
m_{(x,y)}\geq m_{(x+1,y-1)}
}
}\prod_{(i,j)\in L}\frac{1}{m_{(i,j)}^{k_{i}}},
\]
where $L=L(r_1,\ldots,r_s)$.

\end{enumerate}

\end{defn}

We denote by $H$ the set of sequences $k_{1}r_{1}\cdots k_{s}r_{s}k_{s+1}$
with $k_{1},\dots,k_{s+1}\in\mathbb{N}$ and $r_{1},\dots,r_{s}\in\{\up,\down,\open,\close\}$
such that $r_{1}\cdots r_{s}$ is proper and $k_{1}r_{1}\cdots k_{s}r_{s}k_{s+1}$
is admissible.
Note that $\zeta(k_{1}r_{1}\cdots k_{s}r_{s}k_{s+1})$ in Definition
\ref{def_ref} is just a reformulation of the Schur MZVs for $T^{\mathrm{diag}}(\lambda)$.
For example, we have
\[
\zeta\left(
\ {\footnotesize \ytableausetup{centertableaux, boxsize=1.2em}
	\begin{ytableau}
	\none & \none & 2 \\
	\none & 2 & 4 \\
	2 & 3 & 2 \\
	1 & 2 & \none 
	\end{ytableau}}\
\right)=\zeta(2\up4\open2\close3\open2\close1).
\]

Now, Theorem \ref{main1} can be restated as follows:
\begin{thm}\label{main2}

For $k_{1}r_{1}\cdots k_{s}r_{s}k_{s+1}\in H$, we have
\[
\zeta(k_{1}r_{1}\cdots k_{s}r_{s}k_{s+1})=I( P(\bullet\{\up\circ\}^{k_{1}-1}r_{1}\cdots\bullet\{\up\circ\}^{k_{s}-1}r_{s}\bullet\{\up\circ\}^{k_{s+1}-1})),
\]
where $\{\up\circ\}^l$ means $l$-times repetition of $\up\circ$, e.g., $P(\bullet\{\up\circ\}^{2})=P(\bullet\up\circ\up\circ)$.

\end{thm}

For example, we have
\[
\zeta(2\up4\open2\close3\open2\close1)=I\left(P\left(\bullet\up\circ\up\bullet\up\circ\up\circ\up\circ\open\bullet\up\circ\close\bullet\up\circ\up\circ\open\bullet\up\circ\close\bullet\right)\right)
\]
(see the last identity of Example \ref{example1}).

\section{Generalization}


\begin{defn}

For symbols $d_{1},\dots,d_{s+1}\in\{\circ,\bullet\}$, $r_{1},\dots,r_{s}\in\{\up,\down,\open,\close\}$ such that $r_{1}\cdots r_{s}$ is weakly-proper, and real numbers $0\le t_1<\cdots<t_m\le 1$ where $m=E(r_1\cdots r_s)+2$, we shall
define  $J(d_{1}r_{1}\cdots d_{s}r_{s}d_{s+1};t_1,\ldots, t_m)$ by
\[
I(P(d_{1}r_{1}\cdots d_{s}r_{s}d_{s+1}),U),
\]
where $U=(U_{p})_{p\in L}\in\mathcal{U}^{L}$ with $L=L(r_{1},\dots,r_{s})$
is defined by
\[
U_{(x,y)}=[0,1]
\]
 for $1\leq x\leq s$ and 
\[
U_{(s+1,y_j)}=[t_{j},t_{j+1}]
\]
for the $j$-th smallest number $y_j$ of $\{y\in\mathbb{Z}:(s+1,y)\in L\}$.

\end{defn}


Note that
$I( P(\bullet\{\up\circ\}^{k_{1}-1}r_{1}\cdots\bullet\{\up\circ\}^{k_{s}-1}r_{s}\bullet\{\up\circ\}^{k_{s+1}-1}))$ in Theorem \ref{main2} is equal to $J(\bullet\{\up\circ\}^{k_{1}-1}r_{1}\cdots\bullet\{\up\circ\}^{k_{s}-1}r_{s}\bullet\{\up\circ\}^{k_{s+1}-1};0, 1)$.
Furthermore, $J$ satisfies the following recursive relations

\begin{align}\label{relation1}
\begin{split}
&J(d_{1}r_{1}\cdots d_{s}r_{s}d_{s+1}\up e;u_{1},\dots,u_{m}) \\
& \quad=\int_{u_{1}\leq t_{1}\leq\cdots \leq u_{m-1}\leq t_{m-1}\leq u_{m}}\omega_{e}(t_{1})\cdots\omega_{e}(t_{m-1})J(d_{1}r_{1}\cdots d_{s}r_{s}d_{s+1};0,t_{1},\dots,t_{m-1}),\\
&J(d_{1}r_{1}\cdots d_{s}r_{s}d_{s+1}\down e;u_{1},\dots,u_{m}) \\
& \quad=\int_{u_{1}\leq t_{1}\leq\cdots \leq u_{m-1}\leq t_{m-1}\leq u_{m}}\omega_{e}(t_{1})\cdots\omega_{e}(t_{m-1})J(d_{1}r_{1}\cdots d_{s}r_{s}d_{s+1};t_{1},\dots,t_{m-1},1),\\
&J(d_{1}r_{1}\cdots d_{s}r_{s}d_{s+1}\open e;u_{1},\dots,u_{m}) \\
& \quad=\int_{u_{1}\leq t_{1}\leq\cdots \leq u_{m-1}\leq t_{m-1}\leq u_{m}}\omega_{e}(t_{1})\cdots\omega_{e}(t_{m-1})J(d_{1}r_{1}\cdots d_{s}r_{s}d_{s+1};t_{1},\dots,t_{m-1}),\\
&J(d_{1}r_{1}\cdots d_{s}r_{s}d_{s+1}\close e;u_{1},\dots,u_{m}) \\
& \quad=\int_{u_{1}\leq t_{1}\leq\cdots \leq u_{m-1}\leq t_{m-1}\leq u_{m}}\omega_{e}(t_{1})\cdots\omega_{e}(t_{m-1})J(d_{1}r_{1}\cdots d_{s}r_{s}d_{s+1};0,t_{1},\dots,t_{m-1},1)
\end{split}
\end{align}
for $e\in\{\circ,\bullet\}$.
Put
\[
F\binom{t_{0},\dots,t_{j}}{n_{1},\dots,n_{j}}\coloneqq\begin{vmatrix}
1 & \dots & 1\\
t_{0}^{n_{1}} & \dots & t_{j}^{n_{1}}\\
 & \dots\\
t_{0}^{n_{j}} & \dots & t_{j}^{n_{j}}
\end{vmatrix}.
\]

\begin{defn}

Let  $k_{1},\dots,k_{s+1}$ be positive integers and $r_{1},\dots,r_{s}\in\{\up,\down,\open,\close\}$ symbols such that $r_{1}\cdots r_{s}$ is weakly-proper.
Put $j\coloneqq E(r_1\cdots r_s)+2$. Let $t_{1},\dots,t_{j}$
be real numbers such that $0\le t_{1}<\cdots<t_{j}\le1$.
\begin{enumerate}
\item We say that $(k_{1}r_{1}\cdots k_{s}r_{s}k_{s+1};t_1,\ldots,t_j)$ is admissible if and
only if $k_{i}>1$ for all $i$ satisfying
\[
(i=1\lor r_{i-1}\in\{\up,\open\})\land((i=s+1\land t_j=1)\lor r_{i}\in\{\down,\close\}).
\]
\item When  $(k_{1}r_{1}\cdots k_{s}r_{s}k_{s+1};t_1,\ldots,t_j)$ is admissible, we define $\zeta(k_{1}r_{1}\cdots k_{s}r_{s}k_{s+1};t_1,\ldots,t_j)\in\mathbb{R}$ by
\begin{align*}
&\zeta(k_{1}r_{1}\cdots k_{s}r_{s}k_{s+1};t_1,\ldots,t_j)\\
&\quad=\sum_{\substack{(m_{p})_{p\in L}\in\mathbb{N}^{L}\\
m_{(x,y)}<m_{(x+1,y+1)}\\
m_{(x,y)}\geq m_{(x+1,y-1)}
}
}F\binom{t_{1},\dots,t_{j}}{m_{(s+1,y_{1})},\dots,m_{(s+1,y_{j-1})}}\prod_{(x,y)\in L}\frac{1}{m_{(x,y)}^{k_{x}}},
\end{align*}
where $L=L(r_1,\ldots,r_s)$ and $y_i$ is the $i$-th smallest number of $\{y\in\mathbb{Z}:(s+1,y)\in L\}$.
\end{enumerate}

\end{defn}

Note that
$\zeta(k_{1}r_{1}\cdots k_{s}r_{s}k_{s+1})$ in Theorem \ref{main2} is equal to $\zeta(k_{1}r_{1}\cdots k_{s}r_{s}k_{s+1}; 0,1)$.
We denote by $H^{weak}$ the set of sequences $(k_{1}r_{1}\cdots k_{s}r_{s}k_{s+1};t_1,\ldots,t_{j})$
with $k_{1},\dots,k_{s+1}\in\mathbb{N}$ and $r_{1},\dots,r_{s}\in\{\up,\down,\open,\close\}$
such that $r_{1}\cdots r_{s}$ is weakly proper and  $(k_{1}r_{1}\cdots k_{s}r_{s}k_{s+1};t_1,\ldots,t_{j})$
is admissible.
Then the following is a generalization of Theorem \ref{main1}.

\begin{thm}\label{main3}

For $(k_{1}r_{1}\cdots k_{s}r_{s}k_{s+1};t_0,\ldots,t_{e+1})\in H^{weak}$, we have
\begin{align*}
&\zeta(k_{1}r_{1}\cdots k_{s}r_{s}k_{s+1};t_0,\ldots,t_{e+1})\\
&=J(\bullet\{\up\circ\}^{k_{1}-1}r_{1}\cdots\bullet\{\up\circ\}^{k_{s}-1}r_{s}\bullet\{\up\circ\}^{k_{s+1}-1};t_0,\ldots,t_{e+1}),
\end{align*}
where $e=E(r_1\cdots r_s)$.
\end{thm}

\section{Proof of the Main Theorem}


\subsection{Proof of the main theorem}
\begin{lem}\label{lem1}
\begin{itemize}
\item[(\,I\,)]For $0\le u_0<u_1<\cdots<u_j \le1$ and $n_1,\ldots,n_j>0$, we have
\begin{align*} 
 \begin{split}
 &\int_{ u_0<t_{1}<u_{1}<\cdots<t_{j}<u_{j}}
 \begin{vmatrix}
  t_1^{n_1} & \cdots & t_j^{n_1} \\
  & \cdots & \\
  t_1^{n_j} & \cdots & t_j^{n_j} \\
 \end{vmatrix} 
 \frac{dt_{1}}{t_{1}} \cdots \frac{dt_{j}}{t_{j}}
 =\frac{1}{n_1\cdots n_j}
 \begin{vmatrix}
  1 & \cdots &1 \\
  u_0^{n_1} & \cdots & u_j^{n_1} \\
  & \cdots & \\
  u_0^{n_j} & \cdots & u_j^{n_j} 
 \end{vmatrix}.
 \end{split}
\end{align*}\\
\item[(I\hspace{-.01em}I)]For $0\le u_0<\cdots<u_j<1$ and $0\le n_1<\cdots<n_j$, we have
\begin{align*} 
 \begin{split}
  &\int_{ u_{0}<t_{1}<u_{1}<\cdots<t_{j}<u_{j}}
  \begin{vmatrix}
   t_1^{n_1} & \cdots & t_j^{n_1} \\
   & \cdots & \\
   t_1^{n_j} & \cdots & t_j^{n_j} 
  \end{vmatrix}
  \frac{dt_{1}}{1-t_{1}} \cdots \frac{dt_{j}}{1-t_{j}}\\
  &=\sum_{n_1<m_1\le n_2<\cdots\le n_j<m_j}
  \frac{1}{m_1\cdots m_j}
  \begin{vmatrix}
   1 & \cdots & 1\\
   u_0^{m_1} & \cdots & u_j^{m_1} \\
   &\cdots & \\
   u_0^{m_j} & \cdots & u_j^{m_j} 
  \end{vmatrix}.
 \end{split}
\end{align*}\\
\item[(I\hspace{-.15em}I\hspace{-.15em}I)]For $0\le u_0<\cdots<u_j <1$ and $0\le n_1<\cdots<n_{j+1}$, we have
\begin{align*}
 \begin{split}
  &\int_{u_{0}<t_{1}<u_{1}<\cdots<t_{j}<u_{j}}
  \begin{vmatrix}
   t_1^{n_1} & \cdots & t_{j}^{n_1} & 1\\
   & \cdots & \\
   t_1^{n_{j+1}} & \cdots & t_{j}^{n_{j+1}} & 1
  \end{vmatrix}
 \frac{dt_{1}}{1-t_{1}} \cdots \frac{dt_{j}}{1-t_{j}}\\
  &=\sum_{n_1<m_1\le n_2<\cdots<m_{j} \le n_{j+1}}
  \frac{1}{m_1\cdots m_{j}}
  \begin{vmatrix}
   1 & \cdots & 1\\
   u_0^{m_1} & \cdots & u_{j}^{m_1} \\
   &\cdots & \\
   u_0^{m_{j}} & \cdots & u_{j}^{m_{j}} 
  \end{vmatrix}.
 \end{split}
\end{align*}
\end{itemize}
\end{lem}

\begin{proof}
For (\,I\,), we have
\begin{align*} 
\textrm{(L.H.S.)} &=\frac{1}{n_1\cdots n_j}
 \begin{vmatrix}
  u_1^{n_1}-u_0^{n_1} & \cdots & u_j^{n_1}-u_{j-1}^{n_1} \\
  & \cdots & \\
  u_1^{n_j}-u_0^{n_j} & \cdots & u_j^{n_j}-u_{j-1}^{n_j} \\
 \end{vmatrix}\\
 &=\textrm{(R.H.S.)}.
\end{align*}

For (I\hspace{-.15em}I\hspace{-.15em}I), we have

\begin{align*}
\notag
&\textrm{(L.H.S.)} \\
\notag&=
\begin{vmatrix}
\sum_{m_1>n_1} \frac{1}{m_1}(u_1^{m_1}-u_0^{m_1}) & \cdots & \sum_{m_1>n_1} \frac{1}{m_1}(u_{j}^{m_1}-u_{j-1}^{m_1}) & 1 \\
&\cdots & \\
\sum_{m_{j+1}>n_{j+1}} \frac{1}{m_{j+1}}(u_1^{m_{j+1}}-u_0^{m_{j+1}}) & \cdots &\sum_{m_{j+1}>n_{j+1}}  \frac{1}{m_{j+1}}(u_{j}^{m_{j+1}}-u_{j-1}^{m_{j+1}}) & 1
\end{vmatrix}\\
\notag&=
\begin{vmatrix}
1 & \cdots & 1 & 0\\
\sum_{m_1> n_1} \frac{1}{m_1}u_0^{m_1} & \cdots & \sum_{m_1>n_1} \frac{1}{m_1}u_{j}^{m_1} & 1 \\
&\cdots & \\
\sum_{m_{j+1}>n_{j+1}} \frac{1}{m_{j+1}}u_0^{m_{j+1}} & \cdots & \sum_{m_{j+1}>n_{j+1}} \frac{1}{m_{j+1}}u_{j}^{m_{j+1}} & 1 
\end{vmatrix} \\
\notag&=
\begin{vmatrix}
1 & \cdots & 1 \\
\sum_{m_{1}>n_{1}} \frac{1}{m_{1}}u_0^{m_1}-\sum_{m_2>n_2} \frac{1}{m_2}u_0^{m_2} & \cdots & \sum_{m_{1}>n_{1}} \frac{1}{m_{1}}u_{j}^{m_{1}} -\sum_{m_2>n_2} \frac{1}{m_2}u_{j}^{m_2}\\
&\cdots & \\
\sum_{m_{j}>n_{j}} \frac{1}{m_{j}}u_0^{m_{j}} -\sum_{m_{j+1}>n_{j+1}} \frac{1}{m_{j+1}}u_0^{m_{j+1}}& \cdots & \sum_{m_{j}>n_{j}} \frac{1}{m_{j}}u_{j}^{m_{j}} 
-\sum_{m_{j+1}>n_{j+1}} \frac{1}{m_{j+1}}u_{j}^{m_{j+1}}
\end{vmatrix} \\
\notag&=
\begin{vmatrix}
1 & \cdots & 1 \\
\sum_{n_2\ge m_1>n_1} \frac{1}{m_1}u_0^{m_1} & \cdots & \sum_{n_2\ge m_1>n_1} \frac{1}{m_1}u_{j}^{m_1} \\
&\cdots & \\
\sum_{n_{j+1}\ge m_{j}>n_{j}} \frac{1}{m_{j}}u_0^{m_{j}} & \cdots & \sum_{n_{j+1}\ge m_{j}>n_{j}} \frac{1}{m_{j}}u_{j}^{m_{j}}  \\
\end{vmatrix} \\
&=\textrm{(R.H.S.)} .
\end{align*}

Finally, (I\hspace{-.01em}I) is obtained by taking the limit $n_{j+1}\to\infty$ in (I\hspace{-.15em}I\hspace{-.15em}I).
\end{proof}


Put
\[
F\binom{t_{0},\dots,t_{j}}{n_{1},\dots,n_{j}}\coloneqq\begin{vmatrix}
1 & \dots & 1\\
t_{0}^{n_{1}} & \dots & t_{j}^{n_{1}}\\
 & \dots\\
t_{0}^{n_{j}} & \dots & t_{j}^{n_{j}}
\end{vmatrix}
\]
and
\[
X(n_{1},\dots,n_{d+1})\coloneqq\{(m_{1},\dots,m_{d})\in\mathbb{Z}^{d}:n_{1}<m_{1}\leq n_{2}<m_{2}\leq\cdots\leq n_{d}<m_{d}\leq n_{d+1}\}.
\]
\begin{lem} \label{lem2}
Let $0\le n_{1}<\cdots<n_{j}$ be integers and $0\leq u_{0}\leq\cdots\leq u_{j}\leq1$
be real numbers with $j\geq1$. Then, we have the following identities if both sides are defined:
\begin{align*}
\int_{D}F\binom{0,t_{1},\dots,t_{j}}{n_{1},\dots,n_{j}}\frac{dt_{1}}{t_{1}}\cdots\frac{dt_{j}}{t_{j}}&=\frac{1}{n_{1}\dots n_{j}}F\binom{u_{0},\dots,u_{j}}{n_{1},\dots,n_{j}},\\
\int_{D}F\binom{t_{1},\dots,t_{j}}{n_{1},\dots,n_{j-1}}\frac{dt_{1}}{1-t_{1}}\cdots\frac{dt_{j}}{1-t_{j}}&=\sum_{(m_{1},\dots,m_{j})\in X(0,n_{1},\dots n_{j-1},\infty)}\frac{1}{m_{1}\cdots m_{j}}F\binom{u_{0},\dots,u_{j}}{m_{1},\dots,m_{j}},\\
\int_{D}F\binom{0,t_{1},\dots,t_{j}}{n_{1},\dots,n_{j}}\frac{dt_{1}}{1-t_{1}}\cdots\frac{dt_{j}}{1-t_{j}}&=\sum_{(m_{1},\dots,m_{j})\in X(n_{1},\dots n_{j},\infty)}\frac{1}{m_{1}\cdots m_{j}}F\binom{u_{0},\dots,u_{j}}{m_{1},\dots,m_{j}},\\
\int_{D}F\binom{t_{1},\dots,t_{j},1}{n_{1},\dots,n_{j}}\frac{dt_{1}}{1-t_{1}}\cdots\frac{dt_{j}}{1-t_{j}}&=\sum_{(m_{1},\dots,m_{j})\in X(0,n_{1},\dots n_{j})}\frac{1}{m_{1}\cdots m_{j}}F\binom{u_{0},\dots,u_{j}}{m_{1},\dots,m_{j}},\\
\int_{D}F\binom{0,t_{1},\dots,t_{j},1}{n_{1},\dots,n_{j+1}}\frac{dt_{1}}{1-t_{1}}\cdots\frac{dt_{j}}{1-t_{j}}&=\sum_{(m_{1},\dots,m_{j})\in X(n_{1},\dots n_{j+1})}\frac{1}{m_{1}\cdots m_{j}}F\binom{u_{0},\dots,u_{j}}{m_{1},\dots,m_{j}},
\end{align*}
where the integral domain $D=D(u_{0},\dots,u_{j})$ is given by
\[
\{(t_{1},\dots,t_{j}):u_{0}<t_{1}<u_{1}<\cdots<t_{j}<u_{j}\}.
\]
\end{lem}
\begin{proof}
Applying Lemma \ref{lem1}, we obtain the result.
Precisely,
the first, second, third, fourth, and fifth identities follow from (I),
the case $n_1=0$ of (II),
(II),
the case $n_1=0$ of (III),
and (III), respectively.
\end{proof}

By Lemma \ref{lem2}, we have the following recursive relations
\begin{align}\label{relation2}
\begin{split}
\zeta(k_{1}r_{1}\cdots k_{s}r_{s}(k_{s+1}+1);u_{0},\dots,u_{j}) & =\int_{D}\zeta(k_{1}r_{1}\cdots k_{s}r_{s}k_{s+1};0,t_{1},\dots t_{j})\frac{dt_{1}}{t_{1}}\cdots\frac{dt_{j}}{t_{j}},\\
\zeta(k_{1}r_{1}\cdots k_{s}r_{s}k_{s+1}\open1;u_{0},\dots,u_{j}) & =\int_{D}\zeta(k_{1}r_{1}\cdots k_{s}r_{s}k_{s+1};t_{1},\dots t_{j})\frac{dt_{1}}{1-t_{1}}\cdots\frac{dt_{j}}{1-t_{j}},\\
\zeta(k_{1}r_{1}\cdots k_{s}r_{s}k_{s+1}\up1;u_{0},\dots,u_{j}) & =\int_{D}\zeta(k_{1}r_{1}\cdots k_{s}r_{s}k_{s+1};0,t_{1},\dots t_{j})\frac{dt_{1}}{1-t_{1}}\cdots\frac{dt_{j}}{1-t_{j}},\\
\zeta(k_{1}r_{1}\cdots k_{s}r_{s}k_{s+1}\down1;u_{0},\dots,u_{j}) & =\int_{D}\zeta(k_{1}r_{1}\cdots k_{s}r_{s}k_{s+1};t_{1},\dots t_{j},1)\frac{dt_{1}}{1-t_{1}}\cdots\frac{dt_{j}}{1-t_{j}},\\
\zeta(k_{1}r_{1}\cdots k_{s}r_{s}k_{s+1}\close1;u_{0},\dots,u_{j}) & =\int_{D}\zeta(k_{1}r_{1}\cdots k_{s}r_{s}k_{s+1};0,t_{1},\dots t_{j},1)\frac{dt_{1}}{1-t_{1}}\cdots\frac{dt_{j}}{1-t_{j}},
\end{split}
\end{align}
where $j$ is the suitable nonnegative integer (different for each
equation) and the integral domain $D=D(u_{0},\dots,u_{j})$ is the same as in Lemma \ref{lem2}.  
\begin{proof}[Proof of Theorems \ref{main2} and \ref{main3}]

First, we prove Theorem  \ref{main3}.
We use an induction on $q:=k_{1}+\cdots +k_{s+1}$.
The case $q=1$ follows from the following calculation:
\[
\zeta(1;u_{0},u_{1})=\sum_{m=1}^{\infty}\frac{1}{m}F{u_{0},u_{1} \choose m}=\sum_{m=1}^{\infty}\frac{u_{1}^{m}-u_{0}^{m}}{m}=\int_{u_{0}}^{u_{1}}\frac{dt}{1-t}=J(\bullet;u_{0},u_{1}).
\]
The case $q>1$ follows from \eqref{relation1}, \eqref{relation2}, and the induction hypothesis. 
Thus we obtain Theorem \ref{main3}.

Then Theorem \ref{main2} follows from Theorem \ref{main3} since
\[
\zeta(k_{1}r_{1}\cdots k_{s}r_{s}k_{s+1})=\zeta(k_{1}r_{1}\cdots k_{s}r_{s}k_{s+1};0,1)
\]
and
\begin{align*}
&I\left( P\left(\bullet\{\up\circ\}^{k_{1}-1}r_{1}\cdots\bullet\{\up\circ\}^{k_{s}-1}r_{s}\bullet\{\up\circ\}^{k_{s+1}-1}\right)\right)\\
\quad&=J\left(\bullet\{\up\circ\}^{k_{1}-1}r_{1}\cdots\bullet\{\up\circ\}^{k_{s}-1}r_{s}\bullet\{\up\circ\}^{k_{s+1}-1};0,1\right).\qedhere
\end{align*}
\end{proof}

\section{Duality}
\subsection{Preliminary for duality}

We denote by $G$ the set of sequences $d_{1}r_{1}\cdots d_{s}r_{s}d_{s+1}$
with $d_{1},\dots,d_{s+1}\in\{\circ,\bullet\}$ and $r_{1},\dots,r_{s}\in\{\up,\down,\open,\close\}$
such that $r_{1}\cdots r_{s}$ is proper and 
\begin{align*}
&d_{1}r_{1}\notin\{\bullet\down,\ \circ\up\} ,\quad r_{s}d_{s+1}\notin\{\up\bullet,\ \down\circ\},\\
&r_{i}d_{i+1}r_{i+1}\notin\{\up\bullet\down,\ \up\bullet\close,\ \open\bullet\down,\ \open\bullet\close,\ \down\circ\up,\ \down\circ\close,\ \open\circ\up,\ \open\circ\close\}\ (1\le i\le s-1).
\end{align*}
Note that $P(u)$ is admissible 2-labeled poset for $u\in G$.
We define $\theta:H\to G$ by
\[
\theta(k_{1}r_{1}\cdots k_{s}r_{s}k_{s+1})=\bullet\{\up\circ\}^{k_{1}-1}r_{1}\cdots\bullet\{\up\circ\}^{k_{s}-1}r_{s}\bullet\{\up\circ\}^{k_{s+1}-1}.
\]
Then Theorem \ref{main2} says that 
\[
\zeta(u)=I(P(\theta(u)))
\]
 holds for $u\in H$.
 
Now, we define involutions $\tau_{\updownarrow}$, $\tau_{\leftrightarrow}$, 
and $\tau$ of $G$ by
\begin{align*}
\tau_{\updownarrow}(d_{1}r_{1}\cdots d_{s}r_{s}d_{s+1}) & =d_{1}'r_{1}'\cdots d_{s}'r_{s}'d_{s+1}',\\
\tau_{\leftrightarrow}(d_{1}r_{1}\cdots d_{s}r_{s}d_{s+1}) & =d_{s+1}r_{s}''d_{s}\cdots r_{1}''d_{1},
\end{align*}
and $\tau=\tau_{\updownarrow}\circ\tau_{\leftrightarrow}$ where 
\[
\up'=\down,\ \down'=\up,\ \open'=\open,\ \close'=\close,
\]
\[
\up''=\down,\ \down''=\up,\ \open''=\close,\ \close''=\open,
\]
\[
\bullet'=\circ,\ \circ'=\bullet.
\]
Then by definition, we have $I(P(\tau_{\updownarrow}(u)))=I(P(\tau_{\leftrightarrow}(u)))=I(P(\tau(u)))=I(P(u))$
for $u\in G$.
Thus,  for ${\bf k},{\bf l}\in H$, we have 
\[
\zeta({\bf k})=\zeta({\bf l})
\]
if $\sigma(\theta({\bf k}))=\theta({\bf l})$ with some $\sigma\in\{\tau_{\updownarrow},\tau_{\leftrightarrow},\tau\}$.
Therefore, there are possibly 3-types of dualities coming from
$\tau_{\updownarrow}$, $\tau_{\leftrightarrow}$, or $\tau$.
Note that
\begin{align}
\theta(H)=\{d_{1}r_{1}\cdots d_{s}r_{s}d_{s+1}\in G\mid\text{if \ensuremath{d_{i}=\circ}, then \ensuremath{i>1} and \ensuremath{r_{i-1}=\up}}\}.\label{theta}
\end{align}
In later subsections, we will see the followings:
\begin{itemize}
\item the duality for $\tau_{\updownarrow}$ gives no relations,
\item the duality for $\tau_{\leftrightarrow}$ gives new identities for
Schur MZVs (Corollary \ref{cor51}),
\item the duality for $\tau$ gives an alternative proof of Ohno and Nakasuji's
duality theorem (Corollary \ref{tau}).
\end{itemize}
\subsection{Nonexistence of the duality relations coming from $\tau_{\updownarrow}$}
From the definition of $\theta(H)$, we have $ d_1=\bullet$ for $d_{1}r_{1}\cdots d_{s}r_{s}d_{s+1}\in\theta(H)$ and $d_1=\circ$  for $d_{1}r_{1}\cdots d_{s}r_{s}d_{s+1}\in\tau_{\updownarrow}(\theta(H))$.
Thus, unfortunately, there is no pair $({\bf k},{\bf l})\in H^{2}$
satisfying $\tau_{\updownarrow}(\theta({\bf k}))=\theta({\bf l})$.
\subsection{A duality relation coming from $\tau_{\leftrightarrow}$}
By \eqref{theta}, we have
\begin{align*}
&\theta(H)\cap\tau_{\leftrightarrow}(\theta(H))\\
&\quad=\{d_{1}r_{1}\cdots d_{s}r_{s}d_{s+1}\in G\mid\text{if \ensuremath{d_{i}=\circ}, then \ensuremath{1<i\leq s}, \ensuremath{r_{i-1}=\up,} and \ensuremath{r_{i+1}=\down}}\}.
\end{align*}
Therefore,
\begin{align*}
T&:=\theta^{-1}(\theta(H)\cap\tau_{\leftrightarrow}(\theta(H)))\\
&=\{k_{1}r_{1}\cdots k_{s}r_{s}k_{s+1}\in H\mid\text{if \ensuremath{k_{i}\neq1}, then \ensuremath{i\leq s}, \ensuremath{k_{i}=2}, and }r_{s}=\down\}.
\end{align*}
For ${\bf k}\in T$, we define the $\tau_{\leftrightarrow}$-dual index
${\bf l}$ by ${\bf l}=\theta^{-1}(\tau_{\leftrightarrow}(\theta({\bf k})))$.
\begin{cor}[Duality by $\tau_{\leftrightarrow}$]\label{cor51}
Let ${\bf k}$ be an element of $T$ and ${\bf l}$ its $\tau_{\leftrightarrow}$-dual
index. Then we have
\[
\zeta({\bf k})=\zeta({\bf l}).
\]
\end{cor}
\begin{ex}
The $\tau_{\leftrightarrow}$-dual index of $(2\down1\down1)$ is
$(1\up2\down1)$. Thus $\zeta(2\down1\down1)=\zeta(1\up2\down1)$,
i.e.,
\begin{align*}
\zeta\left(
\ {\footnotesize \ytableausetup{centertableaux, boxsize=1.2em}
	\begin{ytableau}
	 1 &  1 & 2
	\end{ytableau}}\
\right)
=
\zeta\left(
\ {\footnotesize \ytableausetup{centertableaux, boxsize=1.2em}
	\begin{ytableau}
	 \none &  1 \\
	 1 & 2 
	\end{ytableau}}\
\right).
\end{align*}
The $\tau_{\leftrightarrow}$-dual index of 
\[
(1\open1\up1\up1\open2\down1\close2\down2\down1\close1\down1)
\]
is
\[
(1\up1\open2\down2\down1\open2\down1\close1\down1\down1\close1).
\]
Thus
\[
\zeta(1\open1\up1\up1\open2\down1\close2\down2\down1\close1\down1)=\zeta(1\up1\open2\down2\down1\open2\down1\close1\down1\down1\close1),
\]
i.e.,
\begin{align*}
\zeta\left(
\ {\footnotesize \ytableausetup{centertableaux, boxsize=1.2em}
	\begin{ytableau}
	\none & \none & \none & \none & \none &  1 & 1\\
	\none & \none & \none & \none & \none &  1 & 1\\
	\none & \none & \none & 1       &  2      &  1 & 1\\
	\none & 1       & 2       & 2       &  1      &  2 & 1\\
	1       & 1       & 1       & 2       &  2      &  1 & 2
	\end{ytableau}}\
\right)
=
\zeta\left(
\ {\footnotesize \ytableausetup{centertableaux, boxsize=1.2em}
	\begin{ytableau}	
	\none & \none & \none & \none & \none & \none &  \none & 1\\
	\none & \none & 1       & 2       & 1       & 2       &  2       & 1\\
	1       & 1       & 1       & 1       & 2       & 1       &  2       & 2\\
	1       & 1       & 1       & 1       & 1       & 2       & \none  & \none
	\end{ytableau}}\
\right).
\end{align*}
\end{ex}
\subsection{A duality relation coming from $\tau$}
By \eqref{theta}, we have
\begin{align*}
&\theta(H)\cap\tau(\theta(H))\\
&\quad=\{d_{1}r_{1}\cdots d_{s}r_{s}d_{s+1}\in G\mid\text{if \ensuremath{d_{i}=\circ}, then \ensuremath{1<i} and \ensuremath{r_{i-1}=\up},}\\
&\qquad\qquad\qquad\qquad\qquad\qquad\ \ \text{if \ensuremath{d_{i}=\bullet}, then \ensuremath{i\leq s} and \ensuremath{r_{i+1}=\up}}\}.
\end{align*}
Therefore,
\begin{align*}
T'&:=\theta^{-1}(\theta(H)\cap\tau(\theta(H)))\\
&=\{k_{1}r_{1}\cdots k_{s}r_{s}k_{s+1}\in H\mid k_{s+1}\ge2 \text{ and [\ensuremath{r_{i}=\up} or }k_{i}\ge2 \text{ for each $1\le i\le s$}]\}.
\end{align*}
For ${\bf k}\in T'$, we define the $\tau$-dual index
${\bf l}$ by ${\bf l}=\theta^{-1}(\tau(\theta({\bf k})))$, or equivalently, for
\[
{\bf k}=(1\up)^{a_1-1}(b_1+1)r_1 \cdots(1\up)^{a_{j-1}-1}(b_{j-1}+1) r_{j-1}(1\up)^{a_j-1}(b_j+1),
\]
the $\tau$-dual index
${\bf l}$ of {\bf k} is defined by
\[
{\bf l}=(1\up)^{b_{j}-1}(a_{j}+1)\widetilde{r_{j-1}} \cdots  (1\up)^{b_{2}-1}(a_{2}+1)\widetilde{r_1}(1\up)^{b_1-1}(a_1+1),
\]
where $a_1,b_1,\ldots,a_j,b_j\ge1$, $r_{1},\dots,r_{j-1}\in\{\up,\down,\open,\close\}$, and $\widetilde{\up}=\up,\ \widetilde{\down}=\down,\ \widetilde{\open}=\close,\ \widetilde{\close}=\open$.
\begin{cor}[Ohno-Nakasuji \cite{ON21}; Duality by $\tau$]\label{tau}
Let ${\bf k}$ be an element of $T'$ and ${\bf l}$ its $\tau$-dual
index. Then we have
\[
\zeta({\bf k})=\zeta({\bf l}).
\]
\end{cor}
\begin{rem}
Corollary \ref{tau} is equivalent to the duality theorem given by Ohno and Nakasuji \cite{ON21}.
\end{rem}
\begin{ex}
The $\tau$-dual index of $(1\up2\open2\close4)$ is
$(1\up1\up2\open2\close3)$. Thus $\zeta(1\up2\open2\close4)=\zeta(1\up1\up2\open2\close3)$,
i.e.,
\begin{align*}
\zeta\left(
\ {\footnotesize \ytableausetup{centertableaux, boxsize=1.2em}
	\begin{ytableau}
	 \none & 1\\
	  2      & 2\\
	  4      & 2
	\end{ytableau}}\
\right)
=
\zeta\left(
\ {\footnotesize \ytableausetup{centertableaux, boxsize=1.2em}
	\begin{ytableau}
	 \none & 1\\
	 \none & 1\\
	  2      & 2\\
	  3      & 2
	\end{ytableau}}\
\right).
\end{align*}
\end{ex}

\section*{Acknowledgements}
This work was supported by JSPS KAKENHI Grant Numbers
 JP18K13392, JP19K14511, JP22K03244, and JP22K13897. 
This work was also supported by 
``Grant for Basic Science Research Projects from The Sumitomo Foundation'' and by ``Research Funding Granted by
The University of Kitakyushu''.


\end{document}